\documentclass[12pt]{article}
\usepackage[utf8]{inputenc}
\usepackage{a4wide}
\usepackage{graphicx}
\usepackage{caption}
\usepackage{subcaption}
\usepackage{placeins}
\usepackage[left=2cm,right=2cm,top=2cm]{geometry}
\usepackage{amsmath}
\usepackage{amsthm}
\usepackage{amsmath} 
\usepackage{amssymb}
\usepackage[]{hyperref}
\usepackage{wrapfig}
\usepackage[percent]{overpic} 
\newtheoremstyle{styl1}
  {13pt}
  {13pt}
  {}
  {}
  {\itshape}
  {.}
  {.5em}
  {}
\theoremstyle{styl1}

\newtheorem{definition}{Definition}
\newtheorem{example}[definition]{Example}
\newtheorem{examples}[definition]{Examples}

\newtheorem{remark}[definition]{Remark}
\newtheorem{remarks}[definition]{Remarks}
\newtheorem{lemma}[definition]{Lemma}
\newtheorem{theorem}[definition]{Theorem}
\newtheorem{proposition}[definition]{Proposition}
\newtheorem{corollary}[definition]{Corollary}

\usepackage{tikz}
\usetikzlibrary{matrix,arrows,decorations.pathmorphing}
\usepackage{tikz-cd}

\newcommand{\gGNS}{GNS }
\newcommand{\tft}{topological field theory }
\newcommand{\tfts}{topological field theories }

\newcommand{\R}{\mathbb{R}}
\newcommand{\C}{\mathbb{C}}
\newcommand{\Ca}{\mathcal{C}}
\newcommand{\K}{\mathbb{K}}

\newcommand{\Hom}{\text{Hom}}
\newcommand{\End}{\text{End}}
\newcommand{\cob}{\text{cob}}
\newcommand{\Set}{\text{Set}}
\newcommand{\lR}{\text{lR}}
\newcommand{\rR}{\text{rR}}
\newcommand{\vs}{\text{vect}_\K}
\newcommand{\RM}[1]{\MakeUppercase{\romannumeral #1{.}}}

\renewcommand{\SS}{\mathbb{S}}
\newcommand{\DD}{\mathbb{D}}
\begin{document}

\thispagestyle{empty}
\begin{flushright}
   {\sf ZMP-HH/17-12}\\
   {\sf Hamburger$\;$Beitr\"age$\;$zur$\;$Mathematik$\;$Nr.$\;$652}\\[2mm]
   March 2017
\end{flushright}
\vskip 2.0em

\begin{center}{\bf \Large
A GNS\ construction of three-dimensional abelian \\[6pt]
Dijkgraaf-Witten theories}

\vskip 18mm

 {\large \  \ Lukas M\"uller\,$^{\,a  \quad}$
 and \quad
 Christoph Schweigert\,$^{\,b}$
 }

\vskip 12mm

 \it$^a$
Department of Mathematics, 
Heriot-Watt University \\
Colin Maclaurin Building, Riccarton, Edinburgh EH14 4AS, U.K. \\
and Maxwell Institute for Mathematical Sciences, Edinburgh, U.K.

 \it$^b$
 Fachbereich Mathematik, \ Universit\"at Hamburg\\
 Bereich Algebra und Zahlentheorie\\
 Bundesstra\ss e 55, \ D\,--\,20\,146\, Hamburg

\end{center}

\vskip 3.2em

\noindent{\sc Abstract}\\[3pt]
We give a detailed account of the so-called ``universal construction'' that
aims to extend invariants 
    of closed manifolds, possibly with additional structure,
to topological field theories and show that 
it amounts to a generalization of the GNS construction.
We apply this construction to an invariant defined in terms of the 
groupoid cardinality
of groupoids of bundles to recover Dijkgraaf-Witten theories, including the
vector spaces obtained as a linearization of spaces of principal bundles.


\section{Introduction}

The Gelfand-Naimark-Segal (GNS) construction associates to a $C^\star$ algebra
$A$ and a state on $A$ a Hilbert space. Similar constructions work in
a purely algebraic setting, and it has been known for a long time
\cite[p.6]{Kerler:1995jw}\cite[p.32]{HomologyTQFT} that the construction of topological field theories
from invariants of closed manifolds with links can be understood in this
way. A Topological field theories is a symmetric monoidal functors from a
category of cobordisms to 
a symmetric monoidal category, say 
vector spaces. The invariants of links
in closed manifolds have various sources. One of them is the Kauffman
bracket; the corresponding three-dimensional topological field theory
has been constructed in \cite{BHMV}. Indeed, our general construction
in section \ref{sec:GNS} of this note is inspired by \cite{BHMV} and many
results in section \ref{sec:GNS} generalize results in \cite{BHMV}.

Heuristically, the invariant for closed manifolds can be seen as the
result of the evaluation of a path integral. 
   In the simplest case of vanishing action, 
   the path integral should count configurations. 
For gauge theories 
based on finite gauge groups, so-called Dijkgraaf-Witten theories
\cite{DW90} \cite{freed1993}, these configurations are finite groupoids; 
counting then means to determine the groupoid cardinality of this groupoid. In
section \ref{sec:DW}, we explicitly deal with Dijkgraaf-Witten theories and 
exhibit a clear relation between
groupoid cardinalities of bundles on three-manifolds with Wilson lines (or, more precisely, ribbon links) 
and linearizations of spaces of $G$-bundles on two-manifolds.

Our results admit several generalizations, including to theories in higher
dimensions and to topological field theories with values in a monoidal
category of modules over a commutative ring. Our results should 
also pave the way towards a more interesting and
challenging generalization, a  categorification of the present
construction, leading to extended topological field theories.

\section{The universal construction as a GNS\ construction}
\label{sec:GNS}

In this section, we present a general formulation of the GNS construction
that is tailored to the construction of topological field theories from
invariants of manifolds.

In a first step, we associate to a category $\Ca$ and an 
object $O\in \Ca $ two functors to the category $\vs$ of
$\K$-vector spaces, where $\K$ is an arbitrary field,
 \[
\begin{tikzcd}
\mathcal{F}_O \colon \; \;  \mathcal{C} \arrow{rr}{\text{Hom}(O, \cdot )}
& & \Set \arrow{rr}{\K[\cdot ]} & & \vs  \\
\end{tikzcd}\; 
\]
and, dually, 
\[
\begin{tikzcd}
\mathcal{F}^{co}_O \colon \; \;  \mathcal{C}^{\mathrm{opp}}\arrow{rr}{\text{Hom}(\cdot
, O )} & & \Set \arrow{rr}{\K[\cdot ]} & & \vs  \\
\end{tikzcd}\; .
\]
Here, $\K[\cdot]\colon \Set \rightarrow \vs$ is the functor that
assigns to a set the $\K$-vector space freely generated over
the set. As an illustrative example inspired by
\cite{BHMV}, the reader might 
keep in mind the example where $\Ca$ is a category
of cobordisms and $O= \emptyset $. Then $\text{End}_\Ca(O)$
are closed manifolds, possibly with additional structure, e.g.
embedded links.
In this situation, important examples of maps of sets 
$I \colon \text{End}_\Ca(O) 
\rightarrow \K$ are invariants of manifolds with embedded
links. In general, we call a map of sets
$I \colon \text{End}_\Ca(O) \rightarrow \K$
a state rooted in the object $O$.

A choice of a state $I$ rooted in $O$
defines for every object $c\in \Ca $ a bilinear 
pairing   
$$\begin{array}{rrll}
(\cdot,\cdot)_c \colon & \mathcal{F}_{O}(c)\otimes_\K 
\mathcal{F}^{co}_{O}(c)&\rightarrow& \K \\
& \delta_f \otimes \delta_g &\mapsto& I(g \circ f), 
\end{array}$$
where $\{\delta_f \mid f\in \Hom(O,c) \}$ and $\{\delta_g \mid g\in \Hom(c,O) \}$ are the canonical bases of the freely generated vector
spaces $\mathcal{F}_{O}(c)$ and $\mathcal{F}^{co}_{O}(c)$, 
respectively. (If $\K$ is the field of complex numbers, 
a sesquilinear pairing can be constructed as well.) 
In general, these pairings are degenerate with a left 
radical $\lR_c$ and right radical $\rR_c$. We consider
the quotients 
\[\mathcal{F}_{O,I}(c) := \mathcal{F}_O(c) / \lR_c  
\text{  and  }
\mathcal{F}^{co}_{O,I}(c) := \mathcal{F}^{co}_O(c) / 
\rR_c  \text{ for all }
c \in \Ca \]
and denote the induced non-degenerate pairing between the
vector spaces $\mathcal{F}_{O,I}(c)$ and $\mathcal{F}^{co}_{O,I}(c)$ by $\langle\cdot , \cdot \rangle_c$.

\begin{lemma}
For any category $\Ca$ and any state 
$I \colon \text{End}_\Ca(O) \rightarrow \K$, we obtain
well-defined functors 
$\mathcal{F}_{O,I}:\Ca\to\vs$ and $\mathcal{F}^{co}_{O,I}:
\Ca^{\mathrm{opp}}\to\vs$. 
\end{lemma}

\begin{proof}
It remains to be shown that $\mathcal{F}_{O,I}$ and $\mathcal{F}^{co}_{O,I}$ are well-defined on morphisms. We present the proof for $\mathcal{F}_{O,I}$. It is enough to show that for all morphisms $h\in \Hom_\Ca(c,c')$ the image of the radical $\lR_c\subset  \mathcal{F}_{O}(c)$ under $\mathcal{F}_{O}[h]$ is contained in $\lR_{c'}\subset  \mathcal{F}_{O}(c')$.  
For $r=\sum_i a_i\delta_{f_i} \in \lR_c$ and all $g\in \text{Hom}(c',O)$,
we find  
\begin{align*}
(\mathcal{F}_{O}[h](r),\delta_g)_{c'}&=\sum_i a_i I(g\circ (h \circ f_i))= \sum_i a_i (\delta_{f_i}, \mathcal{F}^{co}_{O}[h](\delta_g))_c=(r, \mathcal{F}^{co}_{O}[h](\delta_g))_c =0\; . 
\end{align*}
\end{proof}

\begin{definition}
We call the functors $\mathcal{F}_{O,I}$ and $\mathcal{F}^{co}_{O,I}$
a pair of \gGNS functors for the category $\Ca$ and the
state $I:\mathrm{End}(O)\to\K$.
\end{definition}

\begin{remarks}
\begin{enumerate}
\item
Exchanging $\Ca$ and its opposed category $\Ca^{op}$  exchanges
the \gGNS functors
$\mathcal{F}_{O,I}$ and $\mathcal{F}^{co}_{O,I}$. 
For this reason, it usually suffices to prove statements
for the \gGNS functor $\mathcal{F}_{O,I}$.

\item
A $\dagger$-structure on a category $\Ca$ is a involutive 
functor  $\dagger: \Ca \rightarrow \Ca^{op}$ which is the 
identity on objects. 
A state $I\colon \End_\Ca(O) \rightarrow \K$  is compatible with
a $\dagger$-structure on $\Ca$, if $I(f^\dagger)=I(f)$
for all $f \in \text{End}_\Ca (O)$. (In the case of the
field of complex numbers, $\K=\C$, a sesquilinear variant
of the condition, $I(f^\dagger)=I(f)^*$, can be considered
as well.)

For a category with $\dagger$-structure and a compatible state
$I$, we define for all $c\in \Ca$ a pairing by 
$$\begin{array}{rlll}
(\cdot,\cdot)_{c,\dagger } \colon & \mathcal{F}_{O}(c)
\otimes_\K \mathcal{F}_{O}(c)&\rightarrow& \K \\
& \delta_f \otimes \delta_g &\mapsto& I(g^\dagger \circ f)\; . 
\end{array} $$
We then have $\mathcal{F}_{O,I}(c)= \mathcal{F}_{O}(c)/ \lR((\cdot,\cdot)_{c,\dagger })$ and $\mathcal{F}^{co}_{O,I} = \mathcal{F}_{O,I}\circ \dagger$.
\end{enumerate}
\end{remarks}

\begin{examples}
\begin{enumerate}
\item
A $C^\star$-algebra $A$ can be seen as a one object $\C$-linear 
category $\bullet // A$ together with a $\dagger$-structure $ ^\star \colon \bullet  // A \rightarrow (\bullet  //A)^{op}$. A classical state $\tau \colon A
=\mathrm{End}(\bullet) 
\rightarrow \K $ on $A$ is a state
rooted in $\bullet$; it leads to a vector space 
$\mathcal{F}_{\bullet ,\tau}(\bullet)$ endowed with a scalar product $\langle \cdot ,\cdot \rangle_\bullet $. (In general $\mathcal{F}_{\bullet,\tau}(\bullet)$ is not a Hilbert space; by taking its completion, one
obtains a Hilbert space together with an action of $A$. 
This is the the classical Gelfand-Naimark-Segal (GNS) 
construction.) 

The generalization of this example to $C^\star$-categories
is straightforward, see \cite{ghez1985}.
\item
Following \cite{BHMV}, consider a category whose objects are
closed oriented two dimensional manifolds with $p_1$-structure and an even number of embedded arcs and where the morphisms are cobordisms with $p_1$-structure and ribbon links matching the arcs on the boundary. 

Then a state $\langle\;\;\rangle_p$ rooted in
$\emptyset$  can be obtained from the Kauffman bracket and 
a primitive $2p$-th root of unity. In
this context, the role of the \gGNS functors is to provide
vector spaces assigned to codimension-one manifolds. The main theorem of \cite{BHMV} can be formulated using the language of this note 
as: The \gGNS functor corresponding to $\langle\;\;\rangle_p$ 
is symmetric monoidal for $p\geq 3$.

\item
Section \ref{sec:DW} contains a discussion of three-dimensional topological
field theories in terms of \gGNS functors.
\end{enumerate}
\end{examples} 

Inspired by the second example, we now assume that the
category $\Ca$ has
the structure of a
monoidal category $(\mathcal{C},\otimes_\Ca ,\mathbb{I},a,r,l)$. 
Here, $\mathbb{I}$ is the monoidal unit, $a$ the associator
and $r$ and $l$ are unit constraints.
In this case, the monoidal unit $O= \mathbb{I}$ is a natural
choice. Then $\End_\Ca(\mathbb{I})$ has the additional structure
of a unitary monoid.

\begin{proposition}
\label{monoidmorph}
There is a linear isomorphism $\varphi_0 \colon \mathcal{F}_{\mathbb{I},I}(\mathbb{I}) \rightarrow \K $ sending 
$\delta_{\text{id}_{\mathbb{I}}}$ to $1\in \K$,
if and only if  the state 
$I\colon \text{End}_\Ca(\mathbb{I})\rightarrow I$ is a
morphism of unitary monoids.
\end{proposition} 

\begin{proof}
The multiplicativity of $I$ implies for all $f,g\in \text{End}_\Ca (\mathbb{I})$ with $I(g)\neq 0$ the relation  \[ [\delta_f]= \frac{I(f)}{I(g)}[\delta_g] \in \mathcal{F}_{\mathbb{I},I}(\mathbb{I}), \]  where we denote by $[\cdot]$ the equivalence classes in $\mathcal{F}_{\mathbb{I},I}(\mathbb{I})$. $\mathcal{F}_{\mathbb{I},I}(\mathbb{I})$ is not zero dimensional, since $I(\text{id}_\mathbb{I})=1$. We leave the other direction to the reader.    
\end{proof}

We will from now on assume that the state
$I\colon \text{End}_\Ca(\mathbb{I})\rightarrow I$ is a
morphism of unitary monoids. (This assumption typical does
not hold for GNS\ states in quantum mechanics.)
In general, the \gGNS functors are not 
necessarily monoidal; rather a weaker statement holds true:

\begin{theorem}
\label{Theorem monoidal}
Let $(\mathcal{C},\otimes_\Ca ,\mathbb{I},a,r,l)$ be 
monoidal category and $I\colon \text{End}_\Ca(\mathbb{I})\rightarrow \K$ a morphism of unitary monoids.
\begin{enumerate}
\item The natural transformation $\varphi_2 \colon 
\mathcal{F}_{\mathbb{I},I}(\cdot)\otimes_\K 
\mathcal{F}_{\mathbb{I},I}(\cdot)\Rightarrow 
\mathcal{F}_{\mathbb{I},I}(\cdot \otimes_\Ca \cdot )$ defined for $c,c'\in \Ca$ by 
\[ \begin{array}{rlll}
\varphi_{2,c,c'} \colon &\mathcal{F}_{\mathbb{I},I}(c)
\otimes_\K \mathcal{F}_{\mathbb{I},I}(c')&\rightarrow &
\mathcal{F}_{\mathbb{I},I}(c\otimes_\Ca c')\\
&\delta_f\otimes \delta_g &\mapsto &
\delta_{(f\otimes g) \circ r^{-1}}
\end{array}
\] 
is well-defined.
The morphism $\varphi_0$ from proposition
\ref{monoidmorph} and the natural transformation
$\varphi_2$ endow the \gGNS functor $\mathcal{F}_{\mathbb{I},I}$ with the 
structure of a lax monoidal functor. \\
\item The natural transformation $\varphi_2$ is
injective. Furthermore, it is an isomorphism, 
if and only if there exist for all pairs of objects $c,c'\in 
\mathcal{C}$ and any morphism 
$f \in \text{Hom}(\mathbb{I},c\otimes_\mathcal{C} c')$ a 
finite collection of  morphisms $f_{c,i}\in 
\text{Hom}(\mathbb{I},c)$, $f_{c',i}\in \text{Hom}(\mathbb{I},c')$ and 
scalars $a_i\in \mathbb{K}$, such that 
\begin{align}
I(g \circ f)=\sum_i a_i I( g \circ(f_{i,c}\otimes_\Ca f_{i,c'}))
\label{eqationBedingungUniversalTensorfunct}
\end{align} 
for all $g\in \text{Hom}(c\otimes_\mathcal{C} c',\mathbb{I})$.
\end{enumerate}
\end{theorem}

In the definition of $\varphi_2$, we might have alternatively
used $l^{-1}$ instead of $r^{-1}$; both morphisms however
agree on the monoidal unit $\mathbb{I}$.

\begin{proof}
\begin{enumerate}
\item 
We show that the natural transformation 
$\varphi_2$ is well-defined. Consider an arbitrary element 
$r=\sum_i a_i\delta_{f_i}\in \lR_c$ with 
$f_i \in \text{Hom}(\mathbb{I},c)$. For all $g\in \text{Hom}(\mathbb{I},c')$ and $h\in \text{Hom}(c\otimes_\Ca c',\mathbb{I})$ we can calculate
 \begin{align*}
 (\varphi_{2,X,Y}(r\otimes \delta_g),\delta_h)_{c\otimes_\Ca c'}&=\sum a_i I(h \circ (f_i\otimes g)\circ r^{-1})=\sum a_i I(h \circ (\text{id}_X\otimes g)\circ (f_i\otimes \text{id}_Y)\circ r^{-1})\\ &=
\sum a_i I(h \circ (\text{id}_X\otimes g)\circ r^{-1}\circ f_i)=(r,\delta_{h \circ (\text{id}_X\otimes g)\circ r^{-1}})_c=0 \;.
 \end{align*} We can use the same argument for $r'\in \lR_{c'}$. Using linearity this proves that $\varphi_2$ is well-defined. \\
It is straightforward using the definition of a monoidal category to verify that $\varphi_0 $ and $\varphi_2$ endow $\mathcal{F}_{\mathbb{I},I}$ with the structure of a lax monoidal functor. \\
\item 
We define a non degenerate bilinear pairing $\langle \cdot,\cdot\rangle_{c,c'}\colon (\mathcal{F}_{\mathbb{I},I}(c)\otimes \mathcal{F}_{\mathbb{I},I}(c'))\otimes_\K(\mathcal{F}^{co}_{\mathbb{I},I}(c)\otimes_\K \mathcal{F}^{co}_{\mathbb{I},I}(c'))\rightarrow \K$ by 
\[
\langle a\otimes_K b, m \otimes_K n \rangle_{c,c'} = \langle a,m \rangle_c \cdot \langle b,n \rangle_{c'} \; 
\forall a\in \mathcal{F}_{\mathbb{I},I}(c),\; b\in \mathcal{F}_{\mathbb{I},I}(c'),\; m\in \mathcal{F}^{co}_{\mathbb{I},I}(c) \text{ and } n\in \mathcal{F}^{co}_{\mathbb{I},I}(c')  \;,
\] 
for all $c,c'\in \Ca$.
 
The natural transformation $\varphi_2$ and its dual analogue $\varphi_2^{co}\colon \mathcal{F}_{\mathbb{I},I}^{co}(\cdot)\otimes_\K \mathcal{F}_{\mathbb{I},I}^{co}(\cdot) \Rightarrow \mathcal{F}_{\mathbb{I},I}^{co}(\cdot\otimes_\Ca \cdot)$ define a map 

\[ 
\phi_{c,c'}:(\mathcal{F}_{\mathbb{I},I}(c)\otimes_\K \mathcal{F}_{\mathbb{I},I}(c'))
\otimes_\K(\mathcal{F}^{co}_{\mathbb{I},I}(c)\otimes \mathcal{F}^{co}_{\mathbb{I},I}(c'))
\rightarrow \mathcal{F}_{\mathbb{I},I}(c\otimes_\Ca c')\otimes_\K \mathcal{F}^{co}_{\mathbb{I},I}(c\otimes_\Ca c') \;.   
\]
This map preserves the non degenerate bilinear pairing, i.e. 

\[
\langle a\otimes_\K b,m\otimes_K n \rangle_{c,c'}= \langle \varphi_2(a\otimes_\K b),\varphi_2^{co}(m\otimes_\K n) \rangle_{c\otimes_\Ca c'}\;. 
\]

This implies that $\varphi_2$ is injective. 

Equation \eqref{eqationBedingungUniversalTensorfunct} implies 
\begin{align*}
[\delta_f]= \sum_i \varphi_2(a_i [\delta_{f_{ic}}]\otimes_\K [\delta_{f_{ic'}}]),
\end{align*}
 hence $\varphi_2$ is surjective if equation \eqref{eqationBedingungUniversalTensorfunct} holds.
 Obviously, equation \eqref{eqationBedingungUniversalTensorfunct} is true if $\varphi_2$ is surjective.  

\end{enumerate}
\end{proof}

The verification of the condition ensuring that $\varphi_2$ is an
isomorphism can be quite complicated in concrete examples. 

The following definition slightly generalizes the notion of a non-degenerate topological field theory \cite[\RM{3}3.1]{turaev2010quantum} and a cobordism generated functor \cite[p.886]{BHMV}:

\begin{definition}
Let $\Ca$ be a category and $O\in \Ca$. A functor $\mathcal{F}\colon \Ca \rightarrow \vs $ is $O$-exhausted, 
if $\mathcal{F}(c)= \text{span}_\mathbb{K}\{ \text{Im}(\mathcal{F}(f)) \mid f\in \text{Hom}(O,c) \}$ for all $c\in \Ca$. 
\end{definition} 

\gGNS\ functors based on an $O$-rooted state are obviously 
$O$-exhausted. Conversely, a pair of $O$-exhausted functors
related by a non-degenerate bilinear pairing can be recognized
as a pair of \gGNS\ functors:

\begin{proposition}
\label{uniquenesUC}
Let $(\mathcal{C},\otimes_\Ca ,\mathbb{I},a,r,l)$ be a monoidal category and $I\colon \text{End}_\Ca(\mathbb{I})\rightarrow \K$ a morphism of
unitary monoids. Let
$F \colon \mathcal{C}\rightarrow \vs$ and 
$F^{co} \colon \mathcal{C}^{op}\rightarrow \vs$ be 
a pair of $\mathbb{I}$-exhausted functors 
$F \colon \mathcal{C}\rightarrow \vs$ and 
$F^{co} \colon \mathcal{C}^{op}\rightarrow \vs$, which are related
by non-degenerate bilinear pairings \[\widetilde{\langle \cdot,
\cdot  \rangle}_c \colon F(c)\otimes F^{co}(c) \rightarrow \mathbb{K}\] 
for all $c\in \Ca$. Suppose furthermore that there are isomorphisms
$$\varphi\colon \K \rightarrow F(\mathbb{I})
\quad\text{ and }\quad \varphi^{co}\colon \K \rightarrow 
F^{co}(\mathbb{I})$$ 
compatible with the morphism $I$ in the sense that for all
$c\in \Ca$ and all $f\in \text{Hom}(\mathbb{I},c)$ and $g\in
\text{Hom}(c,\mathbb{I})$, we have
\[I(g \circ f)=\widetilde{\langle F(f)[\varphi(1)],
F^{co}(g)[\varphi^{co}(1)] \rangle}_c \,\,. \]  
\begin{enumerate}
\item 
Then there are natural isomorphisms  $\alpha \colon  \mathcal{F}_{\mathbb{I},I} \Rightarrow F$ and $\alpha^{co} \colon \mathcal{F}_{\mathbb{I},I}^{co}\Rightarrow  F^{co} $
 to the \gGNS functors for the state $I$.
 \item
These natural transformations are monoidal, if all functors involved are monoidal, with the isomorphisms $\varphi$ and $\varphi^{co}$ 
as part of the monoidal data.
\end{enumerate}
\end{proposition}
\begin{proof}
We define a natural transformation $\alpha ' \colon  \mathcal{F}_{\mathbb{I},I}\Rightarrow F$ by $\alpha [\delta_f]=F(f)[\varphi (1)]$ and $\alpha_{co}'$ in an analogous way. These maps are surjective and compatible with the bilinear pairing $(\cdot,\cdot)_c$ constructed from $I$. For this reason we get induced natural transformations $\alpha \colon \mathcal{F}_{\mathbb{I},I} \Rightarrow F$ and $\alpha_{co} \colon \mathcal{F}_{\mathbb{I},I,co} \Rightarrow F_{co}$. It is straightforward to check that these are natural isomorphisms. Using the definition of a monoidal functor it is not hard to prove that,
under the assumption stated in the proposition,
these natural transformations are also monoidal.
\end{proof}

A topological field theory is a symmetric monoidal functor
$Z:\cob_{n,n-1}\to\vs$, where $\cob_{n,n-1}$ is a symmetric monoidal
category of cobordisms. To apply \gGNS functors to topological field
theories, it is important to notice that the construction
is compatible with braidings on monoidal categories:

\begin{proposition}
\label{Proposition: Braiding}
Let $(\mathcal{C},\otimes_\Ca ,\mathbb{I},a,r,l,c)$ be a braided monoidal
category with braiding $c$. 
If the \gGNS functor 
    rooted in the monoidal unit
is monoidal, then it is also braided.
\end{proposition}

\begin{proof}
The naturality of the braiding implies that for all
$U,V\in \Ca$ and all  $f_U\in \Hom (\mathbb{I},U)$
and $f_W\in \Hom (\mathbb{I},W)$ the diagram
\[\begin{tikzcd}
U\otimes W \arrow{rr}{c_{U\otimes W}}& & W \otimes U &\\
\mathbb{I}\otimes \mathbb{I}\arrow{u}{f_U\otimes f_W}\arrow{rr}{c_{\mathbb{I}\otimes
\mathbb{I}}} & & \mathbb{I}\otimes \mathbb{I}\arrow{u}{f_W\otimes f_U} &
\\
 & \mathbb{I} \arrow{lu}{r^{-1}} \arrow{ru}{r^{-1}} & 
\end{tikzcd} \] 
commutes. The triangle commutes,
since the braiding is compatible with the right unit constraint.
\cite[Proposition 1]{braidedCat}. With this in mind, 
it is straightforward to check that 
the \gGNS\ functor $\mathcal{F}_{\mathbb{I},I}$
is braided.  
\end{proof}

\begin{remarks}
\begin{enumerate}
\item
The characterization of \gGNS\ functors on (braided) monoidal
categories implies that an $n$-dimensional topological field theory,
i.e.\ a symmetric monoidal functor 
$Z:\cob_{n,n-1}\to\vs$, can be reconstructed from its invariant
on top-dimensional manifolds, if and only if $Z$ is 
$\emptyset$-exhausted. V. Turaev used this uniqueness to prove that every \tft of Reshetikhin-Turaev-type can be reconstructed 
\cite[Chapter \RM{3} Section 3+4 and 
Chapter \RM{4} Lemma 2.1.3]{turaev2010quantum}. 

\item
It is well-known that two-dimensional oriented
\tfts are classified by commutative Frobenius algebras (see for example \cite{kock}).
The topological field theory corresponding to the 
two-dimensional semi-simple Frobenius algebra 
$A=\text{span}(e_1,e_2)$ with multiplication
$e_i\cdot e_j= \delta_{ij}e_i$ and co-unit $\epsilon (e_1)= \epsilon (e_2)
= \lambda \in \C ^\star $ is not $\emptyset$-exhausted:
indeed, the image of any cobordism 
$\emptyset\stackrel M\to \SS^1$ is contained in the one-dimensional
subspace $\C(e_1+e_2)$. (This situation changes if
point defects are included, and one can 
then use 
the uniqueness result from proposition \ref{uniquenesUC} to show that every two-dimensional topological field theory with point defects can be reconstructed using \gGNS functors.)
\end{enumerate}
\end{remarks}

\section{Three dimensional Dijkgraaf-Witten theories}
\label{sec:DW}

We now turn to an application of \gGNS functors:
the construction of three-dimensional oriented
topological field theories. 
From now on, we will work over the field of complex
numbers.
We focus on a specific class of three-dimensional
topological field theories, so-called Dijkgraaf-Witten theories.
Dijkgraaf-Witten theories are gauge theories, based on
a finite gauge group $G$. Our goal is
to obtain them via \gGNS functors from quantities that
are motivated by principles of gauge theory. To obtain
\gGNS functors, we need:

\begin{itemize}
\item
A monoidal category. This will be a monoidal category of
three-dimensional oriented cobordisms. As usual for pure gauge theories,
to have sufficiently many observables at our disposal, we
will have to include Wilson lines.

\item
A state rooted in the monoidal unit, i.e.\ the empty set. 
This state should be thought of as the value of
a ``path integral'' on a closed 3-manifold containing
Wilson loops. For vanishing Lagrangian,
such a value is given
by counting configurations of gauge fields and thus by
a groupoid cardinality of an essentially finite category 
of bundles.
\end{itemize}

We now set up these ingredients carefully. We will assume from
now on that the gauge group $G$ is a finite {\em abelian} group. This assumption
drastically reduces the technical complexity and still 
leads to theories that provide conceptual insight.

To describe the relevant symmetric monoidal category $\cob_{G,3,2}$
including Wilson lines, 
we fix once and for all a standard torus $T^2$ embedded into 
$\R^3$ and a representative $\tau$ for every isomorphism class of principal $G$-bundles over $T^2$. An object $(\Sigma , (\mathfrak{a}_1,\tau_1),\dots ,(\mathfrak{a}_n,\tau_n))$ of $\cob_{G,3,2}$ 
consists of the following data:
\begin{itemize}
\item
A smooth closed oriented two-dimensional manifold 
$\Sigma$,

\item
A finite number of ordered 
{\em pairs} of embedded arcs, described by an embedding  
$\mathfrak{a}_i \colon [-\tfrac{1}{10},\tfrac{1}{10}] \sqcup [-\tfrac{1}{10},
\tfrac{1}{10}] \rightarrow \Sigma$. We require that 
the image of every $\mathfrak{a}_i$, i.e.\ the two arcs in a pair, is 
contained in the same connected component of $\Sigma$.
Each such pair is labelled by a representative $\tau_i$ for a 
principal $G$-bundle over $T^2$.
\end{itemize}

\begin{remarks}
\begin{enumerate}
\item
For convenience, we label the first component of a pair by a $+$ sign and the second by a $-$ sign.
\item
For an object $(\Sigma , (\mathfrak{a}_1,\tau_1),\dots ,(\mathfrak{a}_n,\tau_n))
\in \cob_{G,3,2}$ we define $-(\Sigma , (\mathfrak{a}_1,\tau_1),\dots ,(\mathfrak{a}_n,\tau_n))\in \cob_{G,3,2}$ to be the object consisting of $\Sigma $ with 
reversed orientation and arcs constructed from the arcs $\mathfrak{a}_i$ by reversing the orientation of every arc and exchanging the order 
of every pair of arcs.
\end{enumerate}
\end{remarks} 

To define morphisms from 
$(\Sigma , (\mathfrak{a}_1,\tau_1),\dots ,(\mathfrak{a}_n,\tau_n))$ to 
$(\Sigma' , (\mathfrak{a}'_1,\tau'_1),\dots ,(\mathfrak{a}'_m,\tau'_m))$,
consider smooth compact oriented three-dimensional manifolds $M$ with boundary $\partial M \cong -\Sigma \sqcup \Sigma'$, together with an embedded ribbon link \[l \colon \left( \bigsqcup_{i=1}^{n+m} [-\tfrac{1}{10},\tfrac{1}{10}]\times [0,1] \right) \sqcup \left( \bigsqcup_{j=1}^k [-\tfrac{1}{10},\tfrac{1}{10}]\times \SS^1 \right) \rightarrow M \] such that

\begin{itemize}
\item 
The intervals $[-\tfrac{1}{10},\tfrac{1}{10}]\times 0 $ and $[-\tfrac{1}{10},\tfrac{1}{10}]\times 1 $ are mapped to the negative and positive arcs in the boundary of $M$, respectively. The rest 
of the image of $l$ is contained in the interior of $M$; 
the intersection with the boundary is transversal. 

\item 
The ribbons induce the orientation opposite to the
one given by the arcs.

\item 
Every connected component of the image of $l$
is labelled with a principal $G$-bundle over the standard torus, such that they agree with the labels of the arcs on their boundary.

\item 
The ribbon link respects the pair structure of the boundary
arcs, in the sense that only the following ribbons are allowed between arcs:

\begin{itemize}
\item
A connected component of the link
connecting the two arcs of a pair in the 
outgoing boundary or the two arcs of a pair in the ingoing boundary.

\item
A pair of ribbon links connecting a pair of arcs in the
ingoing boundary to a pair of arcs in the outgoing
boundary.

\end{itemize}
\end{itemize}       
Two such manifolds $(M,l)$ and $(M',l')$ are deemed to
be equivalent if there exists an orientation preserving diffeomorphism relative boundaries
 $\varphi \colon M \rightarrow M'$ mapping $l$ to $l'$ that 
is compatible with all labels.  
 These equivalence classes define the morphisms in $\cob_{G,3,2}$.
\begin{remarks}
\begin{enumerate}
\item
Composition of morphisms is given by gluing along boundaries. 
Composition is only well-defined on equivalence classes
of three-manifolds, since there is no canonical way of 
gluing smooth manifolds and ribbons, but all ways are diffeomorphic. 
\item
$\cob_{G,3,2}$ is a symmetric monoidal category with the disjoint union of manifolds as tensor product and the empty set regarded as a two-dimensional manifold as monoidal unit.
\end{enumerate}
\end{remarks}

As a second ingredient, we need a state rooted at $\emptyset$  which is also a morphism of unitary monoids 
$I\colon \text{End}_{\cob_{G,3,2}}(\emptyset ) \rightarrow \K$, 
i.e.\ a multiplicative invariant for all smooth closed oriented 
three-dimensional manifolds $M$ with embedded labelled ribbon links. 

We consider the simplest possible gauge theories with vanishing
action. (By including a 
a 3-cocycle in $Z^3(G,\C^\times)$,
one could incorporate a
topological Lagrangian; we refrain from doing this in this
short note.) Thus the state should be determined by counting
gauge configurations, i.e.\ by the groupoid
cardinality of an essentially finite groupoid of
$G$-bundles over $M$. To set up our definition, we recall

\begin{definition}[See e.g.\ definition 4 in \cite{groupoidfication}]
Let $\mathcal{G}$ be an essentially finite groupoid and $\pi_0(\mathcal{G})$ the set of isomorphism classes of objects in $\mathcal{G}$. The groupoid cardinality of $\mathcal{G}$ is the positive rational number\begin{align}
|\mathcal{G}| = \sum_{g\in \pi_0(\mathcal{G})}\frac{1}{|\text{Aut}(g)|}\; ,
\end{align} 
where $|\text{Aut}(g)|$ is the cardinality of the automorphism group of 
a representative of $g$.
\end{definition}
\begin{wrapfigure}{r}{0.4\textwidth}
\vspace{-16pt}
\hspace{20pt}
\includegraphics[scale=0.5]{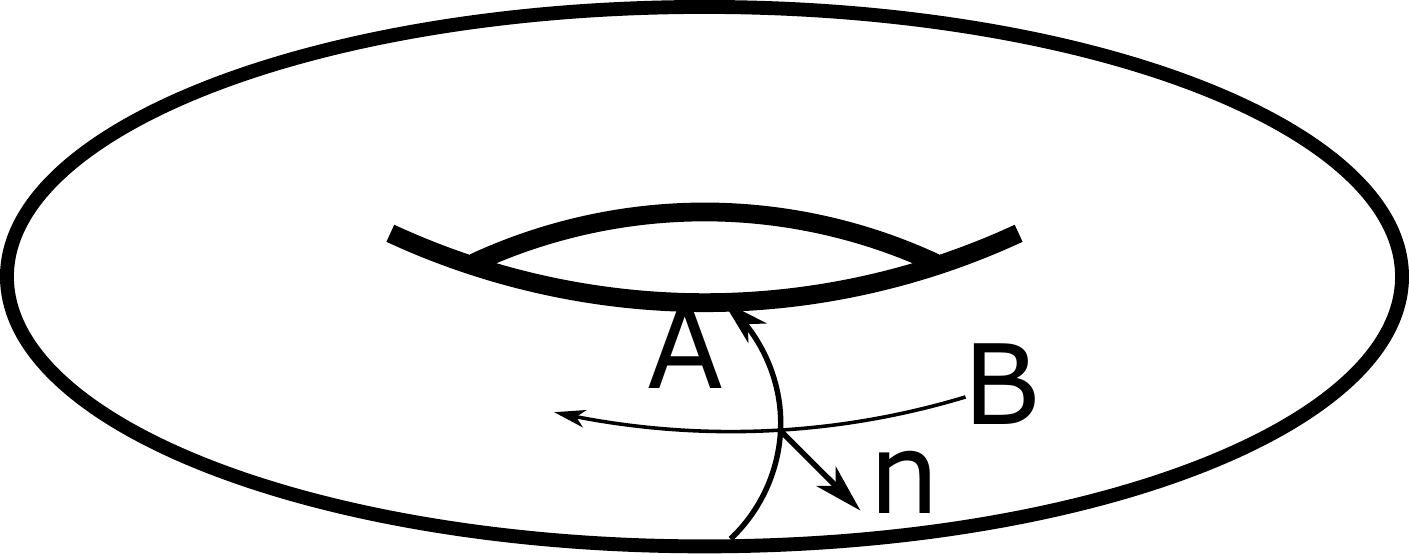}
 \caption{To orient $A$ and $B$ we chose an outwards pointing vector field $n$ on $T^2 \subset \R^3$. }
\label{Fig: Standarttorus}
\vspace{-10pt}
\end{wrapfigure}

We define for manifolds $M$ without Wilson line defects 

\begin{align}
 I_{DW,G}(M):=|\text{Bun}_G(M)|\; .
\label{Equation Invariante ohne Defekte}
\end{align}

The case of three-manifolds with Wilson lines is slightly more involved. The label of a defect Wilson line should fix the behaviour of physical gauge fields ``close" to the defect line. 

To describe how this works in detail, we fix a meridian $A$ and a longitude $B$ on the standard torus in $\R^3$. We orient these circles as pictured in figure \ref{Fig: Standarttorus}.
 An object of $\text{End}_{\cob_{G,3,2}}(\emptyset)$ is represented by a closed manifold $M$, together with a labelled ribbon link $l\colon \sqcup_{i=1}^n
 \SS^1\times [-\tfrac{1}{10},\tfrac{1}{10}]\rightarrow M$. 
\FloatBarrier  
We define $M'_l$  the three manifold with boundary
obtained from $M$ with ``small'' open solid tori around 
the interior of every component of the ribbon link $l$ removed. 
The boundary of $M'_l$ is thus diffeomorphic to a disjoint union of $n$ standard tori. We choose a diffeomorphism $\Phi \colon \sqcup_{i=1}^n T^2 \rightarrow M'_l$
 such that the $B$-cycles of the standard tori are mapped onto boundary components of $l$ and the image of the $A$-cycles is contractable in the solid tori we removed before. This diffeomorphism is unique up to isotopy.

The labels of the components of $l$ determine, via the 
pullback along $\Phi^{-1}$, a principal $G$-bundle $P$ over $\partial M'_l$. We denote by $\text{Bun}_{G,P}(M'_l)$ the groupoid of those
principal $G$-bundles over the three-manifold $M'_l$ 
which restrict to the isomorphism class of $P$ on 
$\partial M'_l$. 

\begin{definition}
\label{dw-state}
Let $G$\ be a finite group.
The Dijkgraaf-Witten state for gauge group $G$ on $\cob_{G,3,2}$
rooted in $\emptyset$ is defined
by its value on the three-manifold with Wilson lines $(M,l)$
\begin{align}\label{Equation Definition I(DW,G)}
 I_{DW,G}(M,l):= |\text{Bun}_{G,P}(M'_l)|. 
\end{align} 
\end{definition}

This extends the 
definition for manifolds without ribbon links in equation 
\eqref{Equation Invariante ohne Defekte}.
One easily checks that this state is multiplicative.
The rest of this paper is 
devoted to the computation of 
the \gGNS 
pair of functors corresponding to the Dijkgraaf-Witten 
state $ I_{DW,G}$. 

We first have to obtain a concrete understanding 
of the state $ I_{DW,G}$. The well-known classification of 
principal $G$-bundles over a connected manifold $M$ is expressed in the
following equivalence of groupoids 

\begin{align}
\text{Bun}_G(M)\cong \Hom(H_1(M),G)// G \cong H^1(M,G)//G ,
\label{Equation Aequivalence bundle Gruppenhomo}
\end{align}
where $G$ acts on a group homomorphism $\varphi : H_1(M)\rightarrow G$ by conjugation.
Since we assumed $G$ to be abelian, we can work with the first
homology group, rather than the fundamental group.
 For an abelian group $G$, we thus
identify isomorphism classes of 
principal $G$-bundles over $T^2$ with elements of 
$G\times G$. We agree that we identify
the first component with the holonomy around the $A$-cycle 
and the second component with the holonomy around
the $B$-cycle. 
\begin{example}
We determine $I_{DW,G}$ for 
an arbitrary $m$ component ribbon link $l$ in $\SS^3$,
labelled by $(a_1,b_1),\dots ,(a_m,b_m)\in G\times G$.
The first homology group of $(\SS^3)'_l$ is $\mathbb{Z}^{m}$. 
We chose $m$ cycles $A_1,\dots, A_m$ going around one component of the ribbon link.
We use the right hand rule to orient them.
These cycles form a basis of $H_1((\SS^3)'_l )$. 

The two homotopic boundaries of every component of the link 
$l$ defines an element 
$B_i$ in $H_1((\SS^3)'_l )$. 
We can express these elements in terms of the basis introduced
above:

\begin{align}
B_i = \sum_{j=1}^m m_{ij}A_j \: ,
\end{align}
with coefficients $m_{ij}\in \mathbb{Z}$. 
 
A group homomorphism $\varphi \colon H_1((\SS^3)'_l )
\rightarrow G $ 
and thus an isomorphism class of $G$-bundles
is completely described by its value on the cycles $A_i$. Compatibility with the labels is equivalent to the following system of equations:
 
\begin{align}
&a_i=\varphi (A_i)  \;\; &\forall i=1,\dots m \label{A_i=a_i} \\
&b_i= \sum_{j=1}^m m_{ij}\varphi (A_j)  \;\; &\forall i=1,\dots m .\label{b_i=...}
\end{align}

 Inserting \eqref{A_i=a_i} into \eqref{b_i=...} leads to a system of $m$ linear conditions in $G$. The invariant corresponding to the link 
$l$ in $\SS^3$ is given by 

\begin{align}
I_{DW,G}(\SS^3,l)= 
\begin{cases} 
\frac{1}{|G|} & \text{if condition \eqref{b_i=...} is satisfied} \\
0 & \text{otherwise} 
\end{cases} \: .
\label{I_DW,G}
\end{align}     

\end{example}

The computation of $I_{DW,G}(\SS^3,l)$ can be involved 
in practice, since one has to express the boundary 
of every component of the link $l$ in terms of the 
generators of $H_1((\SS^3)'_l )$. A more tractable
approach is given by local relations which leave 
the Dijkgraaf-Witten state $I_{DW,G}$ invariant.

\begin{proposition}
\label{proposition: Move1, 2 and 3}
The invariant $I_{DW,G}$ for a labelled ribbon link $l$ in $\SS^3$ does not change under the following moves:
\begin{itemize}
\item
\textbf{1. Move} (figure \ref{fig: Graphical representation move1})\textbf{:} Removing a complete right or left  twist of a component labelled 
by $(a,b)\in G\times G$ and replacing the label $(a,b)$ 
by $(a,b-a)$ or $(a,b+a)$, respectively.
  
\item
\textbf{2. Move}(figure \ref{figureSkein})\textbf{:} Replace two parallel lines corresponding to the same connected component of a ribbon link $l$ in $\SS^3$ by a sum over elements of the form in figure \ref{figureSkein}.      

\item
\textbf{3. Move} (figure \ref{fig: Graphical representation move2})\textbf{:} Removing an over-crossing followed by an under-crossing between a component $(k_1,(b_1,a_1))$ with a second not necessarily different component $(k_2,(b_2,a_2))$ and changing the labels to $(a_1,b_1\pm a_2)$ and $(a_2,b_2\pm a_1)$, where the sign depends on the relative orientation of the two components 
(compare figure \ref{fig: Graphical representation move2}).
\end{itemize}
\end{proposition}

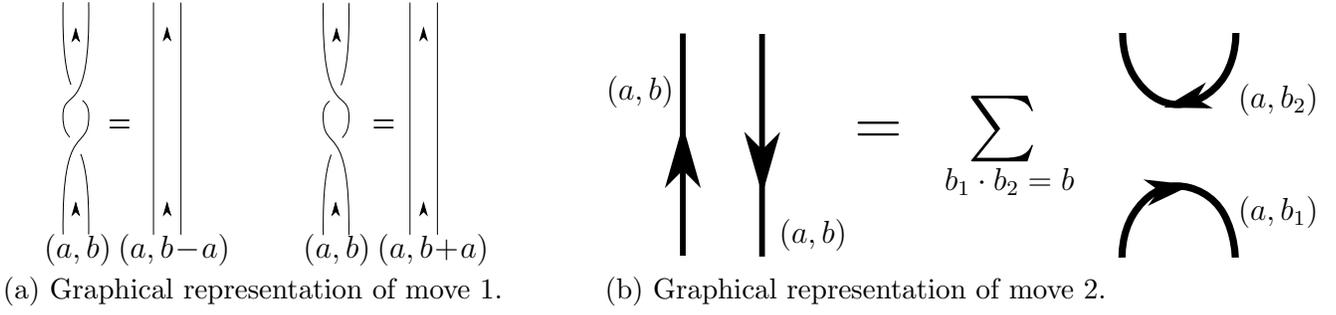
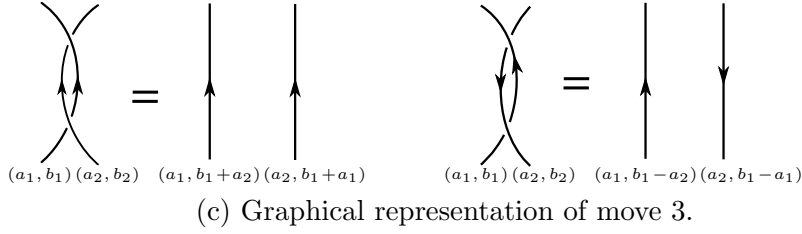
\begin{figure}

\begin{subfigure}[b]{0.4\textwidth}
\begin{overpic}
[scale=0.4]{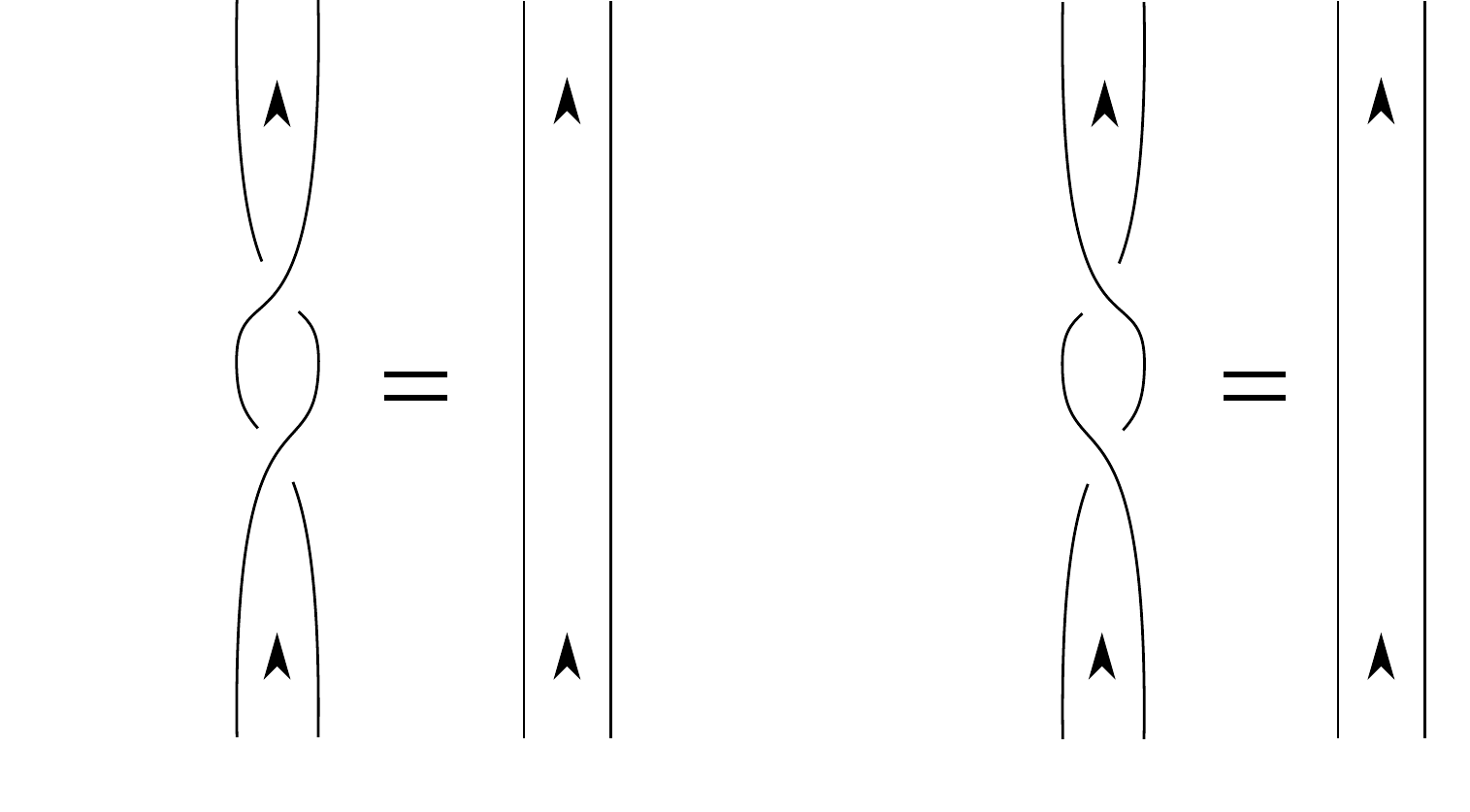}
\put(12,0){$(a,b)$}
\put(28,0){$(a,b\! -\! a)$}
\put(68,0){$(a,b)$}
\put(84,0){$(a,b\! +\! a)$}
\end{overpic}
\caption{Graphical representation of move 1.}
\label{fig: Graphical representation move1}
 \end{subfigure} 
 \hspace{20pt}
\begin{subfigure}[b]{0.4\textwidth}
\begin{overpic}[width=10cm,tics=10,scale=0.7]{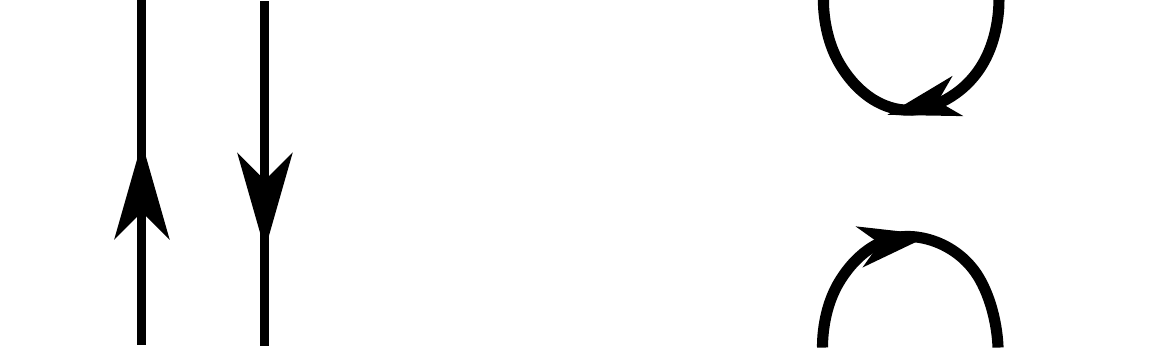}
\put(2,21){$(a,b)$}
\put(25,2){$(a,b)$}
\put(50,15){\Huge{$ \sum$}}
\put(35,15){\Huge{$=$}}
\put(47,9){$b_1\cdot b_2=b $ }
\put(86,5){$(a,b_1)$}
\put(86,20){$(a,b_2)$}
\end{overpic}
\caption{Graphical representation of move 2.}
\label{figureSkein}
\end{subfigure} \vspace{20pt}
\begin{center}
\begin{subfigure}[b]{0.65\textwidth}
\begin{overpic}[width=10cm, tics=10,scale=0.45]{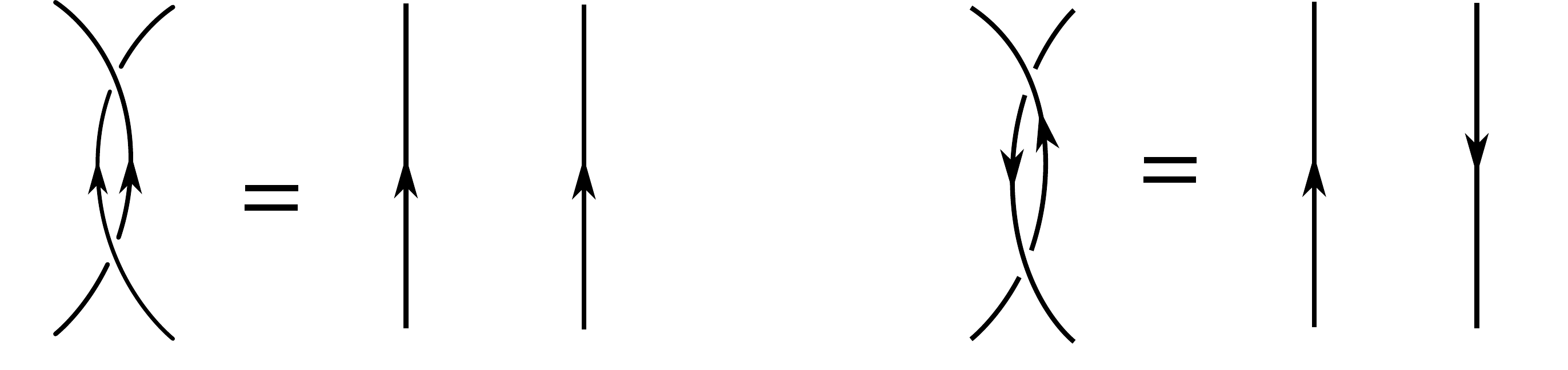}
\put(-1,0){\tiny{$(a_1,b_1)$}}
\put(8,0){\tiny$(a_2,b_2)$}
\put(19,0){\tiny{$(a_1,b_1\! +\! a_2)$}}
\put(33,0){\tiny{$(a_2,b_1\! +\! a_1)$}}
\put(57,0){\tiny{$(a_1,b_1)$}}
\put(66,0){\tiny$(a_2,b_2)$}
\put(77,0){\tiny{$(a_1,b_1\! -\! a_2)$}}
\put(91,0){\tiny{$(a_2,b_1\! -\! a_1)$}}
\end{overpic}
\caption{Graphical representation of move 3.}
\label{fig: Graphical representation move2}
 \end{subfigure}

\caption{Graphical representation of proposition \ref{proposition: Move1, 2 and 3}. The lines represent ribbons in positive blackboard framing.}
\end{center}
\end{figure}

\begin{proof}
\leavevmode
\begin{itemize}
\item 
\textbf{Move 1 and 3:} The first homology group of $(\SS^3)'_l$ does not 
change, if we apply move 1 or 3, but the cycles $B_i$ change. The change of the labels is chosen such that the resulting conditions are invariant.

\item 
\textbf{Move 2:} We denote the link corresponding to a choice of $b_1$ and $b_2$ by $l_{b_1,b_2}$. If the condition \eqref{b_i=...} is satisfied for $l_{b_1,b_2}$ then it also holds for $l$. On the other hand there is exactly one $l_{b_1,b_2}$ satisfying \eqref{b_i=...} for every $l$ satisfying \eqref{b_i=...}.
\end{itemize} 
\end{proof}

\begin{remarks}
\begin{enumerate}
\item
Move 3 can be used to interchange over- and under-crossings. 
This allows us to reduce every ribbon link in $\SS^3$ to 
a collection of simple unknotted links, providing 
an efficient algorithm for computing the value of
the Dijkgraaf-Witten state $I_{DW,G}$ on links in $\SS^3$. 

\item
A result of \cite{BHMV} will allow us to generalize these 
relations to ribbon links in general three dimensional manifolds,
cf.\  corollary 
\ref{Corollary generalisation to arbitrary manifolds}. 
\end{enumerate}
\end{remarks}
 
To explicitly describe the \gGNS functors  
$\mathcal{F}_{\emptyset , I_{DW,G}}$ for the Dijkgraaf-Witten
state $I_{DW,G}$, we determine the vector space 
associated to a torus $T^2$ with the help of a 
Mayer-Vietoris argument.

\begin{proposition}
\label{proposition: Vektorraum T^2}
Fix an untwisted ribbon link $k$ in the standard full torus such that the core of $k$ is homotopic to the $B$-cycle on the boundary.
Each labelling of $k$ gives a morphism in 
$\Hom(\emptyset,T^2)$. The collection of these morphisms
induces a basis of the vector space $\mathcal{F}_{\emptyset , I_{DW,G}}(T^2)$. In particular the dimension of $\mathcal{F}_{\emptyset , I_{DW,G}}(T^2)$ is $|G|\times |G|$.   
\end{proposition}

\begin{proof}
Suppose we are given connected cobordisms
$(M,l)\in \Hom(\emptyset, T^2)$ and $(N,k)\in 
\Hom(T^2, \emptyset)$, together with group 
homomorphisms 
$\varphi_M \colon H_1(M'_l)\rightarrow G$ and 
$\varphi_N \colon H_1(N'_k)\rightarrow G$ which determine
principal $G$-bundles on $M'_l$ and $N'_l$ 
compatible with the labels of $l$ and $k$.\\ 
By the Mayer-Vietoris sequence there is a pushout square 
\begin{equation}
\begin{tikzcd}
H_1(T^2) \ar[r]\ar[d] & H_1((N)'_{k}) \ar[d] \\
H_1((M)'_{l}) \ar[r] & H_1((N \sqcup_{T^2} M)'_{l\sqcup k})
\end{tikzcd} \; .
\end{equation} 

The universal property of pushouts implies that there is a group 
homomorphism 
\[ \varphi_N \colon H_1((N \sqcup_{T^2} M)'_{l\sqcup k})\rightarrow G \]
restricting to $\varphi_M$ and $\varphi_N$, if and only if $\varphi_M|_{T^2} = \varphi_N|_{T^2}$. Every group 
homomorphism corresponding to a bundle over $(N \sqcup_{T^2} M)'_{l\sqcup k}$ compatible with all labels arises in this way.
 This implies that $v=\sum \alpha_i \delta_{(M_i,l_i)} \in \mathcal{F}_{\emptyset}(T^2)$ and $v'=\sum \alpha'_j \delta_{(M'_j,k_j)}\in \mathcal{F}_{\emptyset}(T^2)$ are in the same equivalence class of $\mathcal{F}_{\emptyset , I_{DW,G}}(T^2)$, if for all possible choices of principal $G$-bundles $P$ on $T^2$ 

\begin{align}
\label{equation: Gleichheit}
\sum \alpha_i |\text{Bun}_{G,P\sqcup P(M_i)}((M_i)'_{l_i})|= \sum \alpha'_j |\text{Bun}_{G,P\sqcup P(M'_j)}((M'_j)'_{k_j})|  
\end{align}
holds. Here we denoted by $|\text{Bun}_{G,P\sqcup P(M_i)}((M_i)'_{l_i})|$  the groupoid cardinality of the groupoid consisting of bundles over $(M_i)'_{l_i}$ such that these bundles restrict to the isomorphism class of the bundle defined by $P$ and the labels of links in $M_i$. \\
For a ribbon link $l$ as described in the proposition 
labelled by a principal $G$-bundle $\tau$ there is up isomorphism just one bundle on the boundary which can be extended to $(T^2)_l' $ compatible with the label. This bundle is $\tau $. This implies by \eqref{equation: Gleichheit} that such elements form a basis.    
\end{proof}

To describe the functor on all objects of $\cob_{G,3,2}$, we use a 
general result from Blanchet et al. \cite{BHMV}, 
whose adaptation to $\cob_{G,3,2}$ is straightforward.
To state the result, we first need two definitions:

\begin{definition}[p.\ 889 \cite{BHMV}]
We denote by $\DD^n$ the $n$-dimensional 
closed unit ball.

A functor $Z \colon \cob_{G,3,2}^1 \rightarrow \vs $ with $Z(\emptyset)
\cong \K$ is compatible with surgery, if:
\begin{itemize}
\item[(S0)] 
$ Z(\SS^3)\neq 0 $.
\item[(S1)] 
There is an $\eta \in \K$, such that  $Z(\SS^0\times \DD^3)=\eta Z(\DD^1 \times \SS^2)$.

\item[(S2)] 
The element $Z(\DD^2\times \SS^1)$ lies in the sub vector space generated by ribbon links in the solid torus $-(\SS^1\times \DD^2)$. 
\end{itemize}
\end{definition}

\begin{definition}
Let $M$ be a three-manifold with boundary $\partial M =\Sigma$ 
equipped with the structure of an object in $\cob_{G,3,2}^1$. 
We denote by $\mathcal{L}(M,\Sigma)\subset \mathcal{F}_{\emptyset}(\Sigma)$ the vector space freely generated by the set of equivalence classes of ribbon links in $M$, which describe an element of $\Hom_{\cob_{G,3,2}}(\emptyset,\Sigma)$.
\end{definition}

In \cite{BHMV}, surgery for three-manifolds was used
to show:

\begin{proposition}[Proposition 1.9 of \cite{BHMV}]
\label{Suergery simplification}
If the $\emptyset$-exhausted lax functor from the universal construction $\mathcal{F}_{\emptyset , I} \colon 
\cob_{G,3,2}^1\rightarrow \vs$
for a state $I$ 
is compatible with surgery, then for every $\Sigma \in 
\cob_{G,3,2}$ and connected manifold $M$ with boundary $\Sigma $ the natural map $\pi \colon \mathcal{L}(M,\Sigma)\rightarrow \mathcal{F}_{\emptyset , I}(\Sigma)$ is  surjective. \\  
Furthermore, let $M'$ be any 3-dimensional compact connected manifold with boundary $- \Sigma$, then the kernel of $\pi$ is the same as the left radical of the canonical pairing $(\cdot , \cdot)' \colon \mathcal{L}(M,\Sigma) \otimes \mathcal{L}(M',-\Sigma)\rightarrow \K $.
\end{proposition}

The reader should appreciate that 
the inclusion of co-dimension two defects, 
i.e.\ Wilson lines, 
is crucial for this result.
To apply this result to the \gGNS functor
$\mathcal{F}_{\emptyset , I_{DW,G}}$ for Dijkgraaf-Witten theories,
we need to check its compatibility with surgery:

\begin{lemma}
The \gGNS\ functor $\mathcal{F}_{\emptyset , I_{DW,G}}$ 
is compatible with surgery.
\end{lemma}
\begin{proof}
\leavevmode
\begin{itemize}

\item[(S0)] 
Up to isomorphism, there is just one principal $G$-bundle 
over $\SS^3$, since $H_1(\SS^3)$ is trivial. For this reason 
$I_{DW,G}(\SS^3) = \tfrac{1}{|G|}\neq 0$.

\item[(S1)] 
A Mayer-Vietoris argument similar to the argument in the proof of proposition \ref{proposition: Vektorraum T^2} shows that 
the vector space $\mathcal{F}_{\emptyset,I_{DW,G}}(\SS^2\sqcup \SS^2)$ is one dimensional.
The structure of the groupoid cardinality is crucial for the argument to work. 

The elements in $\mathcal{F}_{\emptyset,I_{DW,G}}(\SS^2\sqcup \SS^2)$ corresponding to the three-manifolds 
$\SS^0\times \DD^3$ and $\DD^1 \times \SS^2$ are not zero, since there exists at least one principal $G$-bundle over every manifold.

\item[(S2)] This requirement holds, since by proposition \ref{proposition: Vektorraum T^2} $\mathcal{F}_{\emptyset,I_{DW,G}}(T^2)$ is generated by ribbons in a full torus.
\end{itemize}
\end{proof}

We can now generalize proposition \ref{proposition: Move1, 2 and 3} to arbitrary manifolds.

\begin{corollary}\label{Corollary generalisation to arbitrary manifolds}
The relations of proposition \ref{proposition: Move1, 2 and 3} hold in any closed oriented three-dimensional manifold.
\end{corollary}

\begin{proof}
Using the compatibility with surgery we can relate the invariant for any ribbon link $l$ in a three dimensional manifold $M$ to a sum of invariants for links in $\SS^3$. We can apply to every term in the sum the relations of proposition \ref{proposition: Move1, 2 and 3}. Reversing the surgery proves the relation for $l$ in $M$.   
\end{proof}

We are now almost in a position to combine
propositions \ref{proposition: Move1, 2 and 3} and  
\ref{Suergery simplification}  to obtain a concrete 
description of the \gGNS functor for the Dijkgraaf-Witten
state $I_{DW,G}$.  
In this description, we will need a specific element which
we construct first.
\begin{figure}
\begin{center}

\begin{overpic}[width=10cm, tics=10,scale=1]{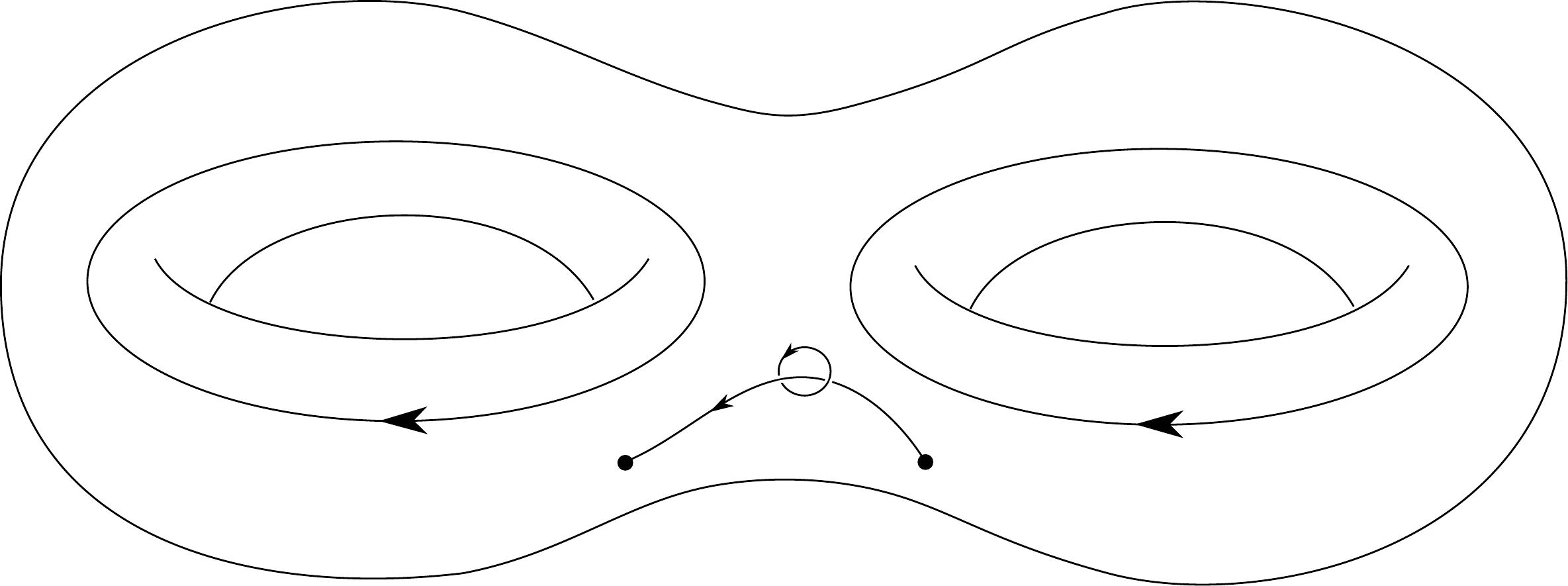}
\put(20,30){\scriptsize$(a_1,\! b_1)$}
\put(70,30){\scriptsize$(a_2,\! b_2)$}
\put(33,6){\scriptsize$(a,\! b)$}
\put(60,6){\scriptsize$(a,\!b)$}
\put(46,16){\scriptsize$(c_1,\! a)$}
\put(3,18){\scriptsize$k_1$}
\put(94,19){\scriptsize$k_2$}
\put(42,6.5){\scriptsize$l_1$}
\put(50,10.5){\scriptsize$g_1$}
\end{overpic}
\end{center}
\caption{An example for a manifold corresponding to the elements defined in remark \ref{Remark: Definition Basis} in the case of a manifold of genus $2$ with a pair of embedded arcs. 
For simplicity, only the core of the ribbon link is drawn.}
\label{Figure: Basis Element}
\end{figure}
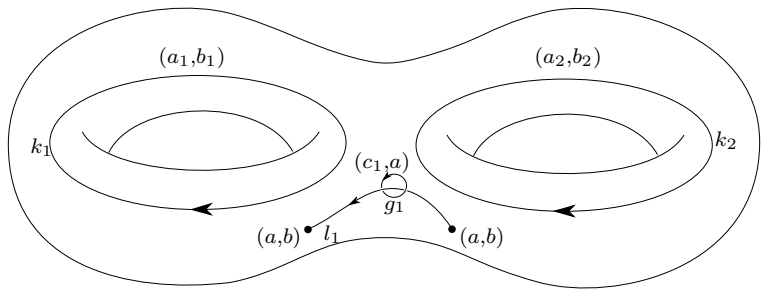
\begin{remark}\label{Remark: Definition Basis}
Let $(\Sigma_g , (\mathfrak{a}_1,\tau_1),\dots ,(\mathfrak{a}_n,\tau_n))\in \cob_{G,3,2}$ be a connected surface of genus $g$ with $n$ pairs of embedded arcs. 
For each such surface, we fix a handle body $H_g$ bounding 
$\Sigma_g$, with the following additional structure (see also figure \ref{Figure: Basis Element}):
\begin{itemize}
\item
A ribbon knot $k_i$ going around every hole of $H_g$, 
a ribbon knot $l_j$ connecting 
every pair of arcs on the boundary and an untwisted ``small'' ribbon $g_j$ going around $l_j$ oriented according to the right hand 
rule. 
\item
The labels of the ribbon knots $l_j$ are determined by the 
labels of the arcs on the boundary $\Sigma_g$. We set the $B$-cycle holonomy of the small ribbons $g_j$ equal to the $A$-cycle 
holonomy of the corresponding ribbon $l_j$. 
\end{itemize}
Any choice of the remaining labels $((a_1,b_1),\dots , (a_g,b_g),c_1 ,\dots , c_n)\in G^{2g+n}$ determines 
a non-zero element. 
We denote the corresponding element of the vector space
$\mathcal{F}_{\emptyset,I_{DW,G}}((\Sigma_g,...))$ by $\delta_{(a_1,b_1),\dots ,(a_g,b_g),c_1,\dots ,c_n}$.   
\label{Basis}
\end{remark}
\begin{theorem}
\leavevmode
Let $G$ be a finite group, $\cob_{G,3,2}$ the cobordism category
with Wilson lines introduced at the beginning of this section
and $I_{DW,G}$ be the 
Dijkgraaf-Witten state introduced in
definition \ref{dw-state}.
\begin{enumerate}
\item 
The \gGNS\ functor 
$\mathcal{F}_{\emptyset,I_{DW,G}} \colon \cob_{G,3,2}
\rightarrow \vs $ based on the Dijkgraaf-Witten state
is a symmetric monoidal functor
and hence defines 
     an oriented
3-2-dimensional topological field theory.
\item 
For any object 
$(\Sigma_g , (\mathfrak{a}_1,\tau_1 ),\dots ,(\mathfrak{a}_n,\tau_n))$
of $\cob_{G,3,2}$, the family of vectors
\begin{align*}
 \left\lbrace \delta_{(a_1,b_1),\dots (a_g,b_g),c_1\dots c_n} \mid((a_1,b_1),\dots (a_g,b_g),c_1\dots c_n)\in G^{2g+n} \right\rbrace
\end{align*}
 form a basis 
of the vector space 
$\mathcal{F}_{\emptyset,I_{DW,G}}((\Sigma_g,(\mathfrak{a}_1,\tau_1),\dots , (\mathfrak{a}_n,\tau_n)))$. 
\end{enumerate}
\end{theorem}

\begin{proof}
\leavevmode
\begin{itemize}
\item[2.] We fix a handle body $H_g$ bounding $\Sigma_g$. We embed $H_g$ into $\SS^3$ and denote the closure of its complement by $M$. Proposition \ref{Suergery simplification} implies \[ \mathcal{F}_{\emptyset , I_{DW,G}}((\Sigma_g,(\mathfrak{a}_1,\tau_1),\dots , (\mathfrak{a}_n,\tau_n))) = \mathcal{L}(H_g,(\Sigma_g,(\mathfrak{a}_1,\tau_1),\dots , (\mathfrak{a}_n,\tau_n)))/ \text{lR}_{(\cdot,\cdot)'}\; , \]
 with \[(\cdot,\cdot)' \colon \mathcal{L}(H_g,(\Sigma_g,(\mathfrak{a}_1,\tau_1),\dots , (\mathfrak{a}_n,\tau_n)))\otimes \mathcal{L}(M,-(\Sigma_g,(\mathfrak{a}_1,\tau_1),\dots , (\mathfrak{a}_n,\tau_n)))\rightarrow \K \: . \] 
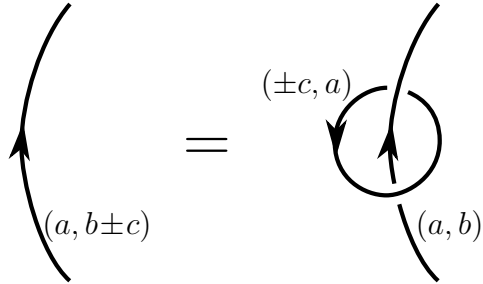
\begin{figure}[b!]
\begin{center}
\begin{overpic}[ scale=1]{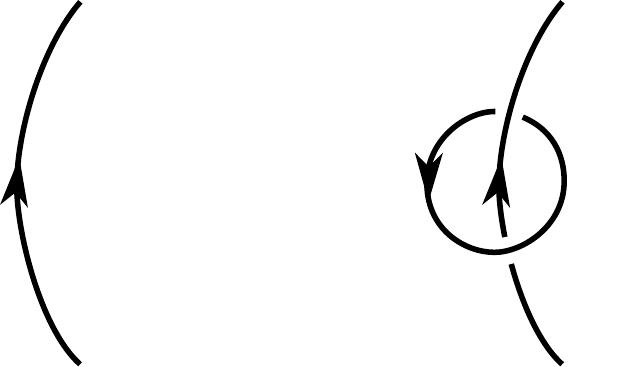}
\put(7,10){$(a,b\! \pm \! c)$}
\put(84,10){$(a,b)$}
\put(36,25){\Huge{$=$}}
\put(52,39){$(\pm c,a)$}
\end{overpic}
 \caption{This relation allows us to change the labels
of the ribbons connecting arcs to the value specified by the arcs after applying move 1, 2 and 3 
inside $H_g$. 
To prove the relation apply move 3 on the right side and use that a contractable untwisted ribbon labelled by $(a,0)$ can be removed without changing the invariant.}
\label{Fig: RelationRing}
\end{center}
\end{figure}
We can use the local relations developed in proposition \ref{proposition: Move1, 2 and 3} inside $H_g$, since the union of $H_g$ and $M$ is $\SS^3$. This changes the labels of the arcs on $\Sigma $. It is possible to compensate this 
by adding a ring around the ribbon
using the relation in figure \ref{Fig: RelationRing}. 
These relations allow us to reduce every ribbon link in $H_g$ to a collection of knots going around every hole and ribbon links connecting the arcs on the boundary with small circles around them. From proposition \ref{proposition: Vektorraum T^2} it follows that we can replace the collection of ribbon knots going around every hole by a sum over ribbon links with just one component going around every hole. This proves that the $\delta_{(a_1,b_1),\dots (a_g,b_g),c_1\dots c_n}$ generated $\mathcal{F}_{\emptyset,I_{DW,G}}((\Sigma_g,(\mathfrak{a}_1,\tau_1),\dots , (\mathfrak{a}_n,\tau_n)))$.

We still have to show that the family
$\left(\delta_{(a_1,b_1),\dots (a_g,b_g),c_1\dots c_n}\right)$ is
linearly independent. To this end,
we introduce for every hole in $M$ ribbons $k^\star_i$ going once around this hole oriented according to the right hand rule with respect to $k_i$, a ribbon knot $l^\star_j$ connecting every pair of arcs on the boundary and an untwisted ``small'' ribbon $g^\star_j$ going around $l^\star_j$. A choice of labels as in definition \ref{Basis} defines elements $\delta^\star_{(a_1,b_1),\dots (a_g,b_g),c_1\dots c_n} \in \mathcal{F}^{co}_{\emptyset , I_{DW,G}}(\Sigma_g,(\mathfrak{a}_1,\tau_1),\dots , (\mathfrak{a}_n,\tau_n))$. We have 
\begin{align}
\label{Parung}
\langle \delta_{(a_1,b_1),\dots (a_g,b_g),c_1\dots c_n},\delta^\star_{(a'_1,b'_1),\dots (a'_g,b'_g),c'_1\dots c'_n}\rangle = \delta_{a_1,b'_1}\cdot \delta_{b_1,a'_1} \cdots \delta_{c_n+c'_n, \mathcal{B}_n}, 
\end{align} where $\mathcal{B}_i$ is the $B$-cycle holonomy of $\tau_i$. Equation \eqref{Parung} implies that the $\delta_{(a_1,b_1),\dots (a_g,b_g),c_1\dots c_n}$ are linearly independent.  

\item[1.] Let $\Sigma_1$ and $\Sigma_2$ be elements of $\cob_{G,3,2}$. 
As a connected 3-manifold boundary
$\Sigma_1 \sqcup \Sigma_2$ we choose a connected sum of handle bodies $M$. We embed the handle bodies into $\SS^3$ and denote the closure of its complement by $N$. Gluing $M$ and $N$ together gives a manifold $S$. By corollary \ref{Corollary generalisation to arbitrary manifolds} we can apply the relations of proposition \ref{proposition: Move1, 2 and 3} also inside of $M$. We can use this to reduce a given ribbon link in $M$ corresponding to a morphism $f$ in $\text{Hom}_{\cob_{G,3,2}}(\emptyset, \Sigma_1 \sqcup \Sigma_2)$ to a link, for which any component is completely contained in one of the handle bodies. For this we need that every pair of labelled arcs is contained in the same connected component of $\Sigma_1 \sqcup \Sigma_2$. By applying a finite number of 1 surgery moves we see that $ f $ is equivalent to a linear combination of disjoint unions of handle bodies in the vector space associated to $\Sigma_1 \sqcup \Sigma_2 $. The statement follows from theorem \ref{Theorem monoidal} and proposition \ref{Proposition: Braiding}. 
\end{itemize}
\end{proof}

\begin{remark}
A basis element $\delta_{(a_1,b_1),\dots (a_g,b_g)}$ for a surfaces $\Sigma_g$ without arcs defines a unique isomorphism class of principal $G$-bundles $[P_{(a_1,b_1),\dots (a_g,b_g)}]\in \pi_0 (\text{Bun}_G(\Sigma_g))$. A representative for this class can be constructed by restricting a bundle compatible with all labels over the handle body $H_g$ with $g$ solid tori removed to $\Sigma_g$. 
This shows that the vector space $\mathcal{F}_{\emptyset, I_{DW,G}}(\Sigma_g)$ can be naturally identified with the linearization of isomorphism classes of flat $G$-bundles over $\Sigma_g$.

Furthermore, we can calculate the transition amplitude corresponding to a morphism $M\in \text{Hom}_{\cob_{G,3,2}}(\Sigma_g, \Sigma_{g'})$ by gluing in the manifolds corresponding to $\delta_{(a_1,b_1),\dots (a_g,b_g)}$ and $\delta^\star_{(a'_1,b'_1),\dots (a'_{g'},b'_{g'})}$. 
This reproduces the known description of transition amplitudes for 
Dijkgraaf-Witten theories in terms of the cardinalities of the groupoid 
of bundles over $M$ restricting to prescribed bundles on the boundary.
\end{remark}

\noindent
{\sc Acknowledgements:}\\[.3em]
CS is partially supported by the Collaborative Research Centre 676 ``Particles,
Strings and the Early Universe - the Structure of Matter and Space-Time'' 
and by the RTG 1670 ``Mathematics inspired by String theory and Quantum 
Field Theory''. 
LM is supported by the studentship 
ST$|$N509099$|$1 from the U.K. Science and 
Technology Facilities Council.

\FloatBarrier 
\bibliography{Quellen}
\end{document}